\documentclass[10pt,english,a4paper]{article}

\usepackage[latin1]{inputenc}
\usepackage[T1]{fontenc}
\usepackage{lmodern}
\usepackage{babel}
\usepackage{amsmath,amsfonts,amssymb,amsthm}
\usepackage{microtype}
\usepackage{hyperref}

\theoremstyle{plain}
\newtheorem{thm}{Theorem}[section]
\newtheorem{prop}[thm]{Proposition}
\newtheorem{cor}[thm]{Corollary}
\newtheorem{lem}[thm]{Lemma}
\newtheorem*{thma}{Theorem A}
\newtheorem*{thmb}{Theorem B}

\theoremstyle{definition}
\newtheorem*{defn}{Definition}
\newtheorem{exmp}[thm]{Example}
\newtheorem{rem}[thm]{Remark}

\makeatletter
\newcommand{\subjclass}[1]{
  \let\@oldtitle\@title
  \gdef\@title{\@oldtitle\footnotetext{2010 \emph{Mathematics Subject Classification.} #1}}
}
\newcommand{\keywords}[1]{
  \let\@@oldtitle\@title
  \gdef\@title{\@@oldtitle\footnotetext{\emph{Key words and phrases.} #1.}}
}
\makeatother

\title{Primeness and primitivity conditions for twisted group $C^*$-algebras}
\author{Tron {\AA}nen Omland \thanks{Partially supported by the Research Council of Norway}}
\date{\today}

\begin{document}

\maketitle

\begin{abstract}
For a multiplier (2-cocycle) $\sigma$ on a discrete group $G$ we give conditions for which
the twisted group $C^*$-algebra associated with the pair $(G,\sigma)$ is prime or primitive.
We also discuss how these conditions behave on direct products and free products of groups.
\end{abstract}

\section*{Introduction}
\addcontentsline{toc}{section}{Introduction}

In this paper, $G$ will always denote a discrete group with identity $e$.
The full group $C^*$-algebra associated with $G$, $C^*(G)$ is simple only if $G$ is trivial, but other aspects of its ideal structure are of interest.
Recall that a $C^*$-algebra is called \emph{primitive} if it has a faithful irreducible representation
and \emph{prime} if nonzero ideals have nonzero intersection.
Primeness of a $C^*$-algebra is in general a weaker property than primitivity.
However, according to a result of Dixmier \cite{Dixmier}, the two notions coincide for separable $C^*$-algebras.

Furthermore, recall what the \emph{icc property} means for $G$ \textemdash{} that every nontrivial conjugacy class is infinite,
and its importance comes to light in the following theorem.

\begin{thma}\label{thma}
The following are equivalent:
\begin{itemize}
\item[(i)] $G$ has the icc property.
\item[(ii)] The group von Neumann algebra $W^*(G)$ is a factor.
\item[(iii)] The reduced group $C^*$-algebra $C^*_r(G)$ is prime.
\end{itemize}
\end{thma}

The equivalence (i) $\Leftrightarrow$ (ii) is a well known result of Murray and von Neumann \cite{Murray-Neumann},
while (i) $\Leftrightarrow$ (iii) is proved by Murphy \cite{Murphy}. 
Murphy also shows that the icc property is a necessary condition for primeness of $C^*(G)$.
Therefore, for the class of discrete groups, primeness and, in the countable case, primitivity, may be regarded as $C^*$-algebraic analogs of factors.
The theorem gives as a corollary that if $G$ is countable and amenable, then primitivity of $C^*(G)$ is equivalent with the icc property of $G$.
Moreover, since amenability of $G$ implies injectivity of $W^*(G)$,
this is also equivalent to $W^*(G)$ being the hyperfinite II$_1$ factor if $G$ is nontrivial,
according to Connes \cite{Connes}.

In the present paper, \hyperref[thma]{Theorem~A} will be adapted to a twisted setting where pairs $(G,\sigma)$
consisting of a group $G$ and a multiplier $\sigma$ on $G$ are considered.
We will show that the reduced twisted group $C^*$-algebra $C^*_r(G,\sigma)$ is prime
if and only if every nontrivial $\sigma$-regular conjugacy class of $G$ is infinite,
and say that the pair $(G,\sigma)$ satisfies \emph{condition~K} if it possesses this property.
It was first introduced by Kleppner \cite{Kleppner},
who proves that this property is equivalent to the fact that the twisted group von Neumann algebra $W^*(G,\sigma)$ is a factor.
The main part of our proof is to show that $(G,\sigma)$ satisfies condition~K if and only if $C^*_r(G,\sigma)$ has trivial center,
and this argument is, of course, inspired by the mentioned works of Kleppner and Murphy.
As a corollary, we get that primeness of the full twisted group $C^*$-algebra $C^*(G,\sigma)$ implies condition~K on $(G,\sigma)$.
The converse is not true in general, but at least holds if $G$ is amenable, as the full and reduced twisted group $C^*$-algebras then are isomorphic.
Thus, if $G$ is countable and amenable, condition~K on $(G,\sigma)$ is equivalent to primitivity of $C^*(G,\sigma)$ by applying Dixmier's result.
This fact is also explained by Packer \cite{Packer-N} with a different approach.
On the other hand, no examples of nonprimitive, but prime twisted group $C^*$-algebras are known,
so it is not clear whether we need the countability assumption on $G$.

In the last two sections we will investigate primeness and primitivity of the twisted group $C^*$-algebras of $(G,\sigma)$
when $G=G_1\times G_2$ and when $G=G_1*G_2$.
The free product case turns out to be easier to handle in general,
since the corresponding $C^*$-algebra always decomposes into a free product of the two components.
For direct products, however, the multiplier $\sigma$ on $G$ can have a 'cross-term'
which makes a $C^*$-algebra decomposition into tensor products impossible.

\bigskip

A significant part of this work,
especially \hyperref[pp]{Section~2}, was accomplished when I was a student at University of Oslo, and is also included in my master's thesis.
I am indebted to Erik B{\'e}dos for his advice, both on the thesis and on the completion of this paper.

I would also like to thank the referee for several useful comments and suggestions.

\section{Twisted group \texorpdfstring{$C^*$}{C*}-algebras}


Let $G$ be a group and $\mathcal{H}$ a nontrivial Hilbert space.
The projective unitary group $PU(\mathcal{H})$ is the quotient of the unitary group $U(\mathcal{H})$ by the scalar multiples of the identity,
that is,
\begin{displaymath}
PU(\mathcal{H})=U(\mathcal{H})/\mathbb{T}1_{\mathcal{H}}.
\end{displaymath}
A projective unitary representation $G$ is a homomorphism $G\to PU(\mathcal{H})$.
Every lift of a projective representation to a map $U:G\to U(\mathcal{H})$ must satisfy
\begin{equation}\label{proj-repr}
U(a)U(b)=\sigma(a,b)U(ab)
\end{equation}
for all $a,b\in G$ and some function $\sigma:G\times G\to\mathbb{T}$.
From the associativity of $U$ and by requiring that $U(e)=1_{\mathcal{H}}$, the identities
\begin{equation}\label{multiplier}
\begin{split}
\sigma(a,b&)\sigma(ab,c)=\sigma(a,bc)\sigma(b,c)\\
&\sigma(a,e)=\sigma(e,a)=1
\end{split}
\end{equation}
must hold for all elements $a,b,c\in G$.

\begin{defn}
Any function $\sigma:G\times G\to\mathbb{T}$ satisfying \eqref{multiplier} is called a \emph{multiplier on $G$},
and any map $U:G\to U(\mathcal{H})$ satisfying \eqref{proj-repr} is called a \emph{$\sigma$-projective unitary representation of $G$ on $\mathcal{H}$}.
\end{defn}

The lift of a homomorphism $G\to PU(\mathcal{H})$ up to $U$ is not unique,
but any other lift is of the form $\beta U$ for some function $\beta:G\to\mathbb{T}$.
Therefore, two multipliers $\sigma$ and $\tau$ are said to be \emph{similar} if
\begin{displaymath}
\tau(a,b)=\beta(a)\beta(b)\overline{\beta(ab)}\sigma(a,b)
\end{displaymath}
for all $a,b\in G$ and some $\beta:G\to\mathbb{T}$.
Note that we must have $\beta(e)=1$ for this to be possible.
We say that $\sigma$ is \emph{trivial} if it is similar to $1$ and call $\sigma$ \emph{normalized} if $\sigma(a,a^{-1})=1$ for all $a\in G$.

Moreover, the set of similarity classes of multipliers on $G$ is an abelian group under pointwise multiplication.
This group is the \emph{Schur multiplier} of $G$ and will henceforth be denoted by $\mathcal{M}(G)$.
Also, we remark that multipliers are often called \emph{$2$-cocycles on $G$ with values in $\mathbb{T}$},
and that the Schur multiplier of $G$ coincides with the second cohomology group $H^2(G,\mathbb{T})$.  \bigskip

Let $\sigma$ be a multiplier on $G$.
We will briefly explain how the operator algebras associated with the pair $(G,\sigma)$ are constructed
and refer to Zeller-Meier \cite{Zeller-Meier} for further details.
First, the Banach $^*$-algebra $\ell^1(G,\sigma)$ is defined as the set $\ell^1(G)$ together with twisted convolution and involution given by
\begin{displaymath}
\begin{split}
(f*_{\sigma}g)(a)&=\sum_{b\in G} f(b)\sigma(b,b^{-1}a)g(b^{-1}a)\\
f^*(a)&=\overline{\sigma(a,a^{-1})}\overline{f(a^{-1})}
\end{split}
\end{displaymath}
for elements $f,g$ in $\ell^1(G)$, and is equipped with the usual $\lVert\cdot\rVert_1$-norm.
\begin{defn}
The full twisted group $C^*$-algebra $C^*(G,\sigma)$ is the universal enveloping algebra of $\ell^1(G,\sigma)$.
Moreover, the canonical injection of $G$ into $C^*(G,\sigma)$  will be denoted by $i_{(G,\sigma)}$ or just $i_G$ if no confusion arise.
\end{defn}
For $a$ in $G$, let $\delta_a$ be the function on $G$ defined by
\begin{displaymath}
\delta_a(b)=\begin{cases} 1 &\text{ if } b=a,\\0 &\text{ if } b\neq a.\end{cases}
\end{displaymath}
Then the set $\{\delta_a\}_{a\in G}$ is an orthonormal basis for $\ell^2(G)$ and generates $\ell^1(G,\sigma)$,
so that for all $a$ in $G$, $i_{(G,\sigma)}(a)$ is the image of $\delta_a$ in $C^*(G,\sigma)$.
The \emph{left regular $\sigma$-projective unitary representation $\lambda_{\sigma}$ of $G$ on $B(\ell^2(G))$} is given by
\begin{displaymath}
(\lambda_{\sigma}(a)\xi)(b)=(\delta_a*_{\sigma}\xi)(b)=\sigma(a,a^{-1}b)\xi(a^{-1}b).
\end{displaymath}
Note in particular that
\begin{displaymath}
\lambda_{\sigma}(a)\delta_b=\delta_a *_{\sigma} \delta_b=\sigma(a,b)\delta_{ab}
\end{displaymath}
for all $a,b\in G$.
Moreover, the integrated form of $\lambda_{\sigma}$ on $\ell^1(G,\sigma)$ is defined by
\begin{displaymath}
\lambda_{\sigma}(f)=\sum_{a\in G}f(a)\lambda_{\sigma}(a).
\end{displaymath}

\begin{defn}
The reduced twisted group $C^*$-algebra and the twisted group von Neumann algebra of $(G,\sigma)$, $C^*_r(G,\sigma)$ and $W^*(G,\sigma)$ are,
respectively, the $C^*$-algebra and the von Neumann algebra generated by $\lambda_{\sigma}(\ell^1(G,\sigma))$, or equivalently by $\lambda_{\sigma}(G)$.
\end{defn}

If $\tau$ is similar with $\sigma$, then in all three cases, the operator algebras associated with $(G,\tau)$ and $(G,\sigma)$ are isomorphic. \bigskip

Moreover, there is a natural one-to-one correspondence between the representations of $C^*(G,\sigma)$ 
and the $\sigma$-projective unitary representations of $G$.
In particular, $\lambda_{\sigma}$ extends to a $^*$-homomorphism of $C^*(G,\sigma)$ onto $C^*_r(G,\sigma)$. 
If $G$ is amenable, then $\lambda_{\sigma}$ is faithful.
However, it is not known whether the converse holds unless $\sigma$ is trivial. \bigskip

Following the work of Kleppner \cite{Kleppner},
an element $a$ in $G$ is called \emph{$\sigma$-regular} if $\sigma(a,b)=\sigma(b,a)$ whenever $b$ commutes with $a$, or equivalently if
\begin{displaymath}
U(a)U(b)=U(b)U(a)
\end{displaymath}
for all $b$ commuting with $a$ whenever $U$ is a $\sigma$-projective unitary representation of $G$.
If $\sigma$ and $\tau$ are similar multipliers on $G$, it is easily seen that $a$ in $G$ is $\sigma$-regular if and only if it is $\tau$-regular.
Furthermore, if $a$ is $\sigma$-regular, then $cac^{-1}$ is $\sigma$-regular for all $c$ in $G$,
and thus the notion of $\sigma$-regularity makes sense for conjugacy classes \cite[Lemma~3]{Kleppner}.
The following theorem may now be deduced from \cite[Lemma~4]{Kleppner}.
\begin{thmb}\label{thmb}
Let $\sigma$ be a multiplier on $G$. Then the following are equivalent:
\begin{itemize}
\item[(i)] Every nontrivial $\sigma$-regular conjugacy class of $G$ is infinite.
\item[(ii)] $W^*(G,\sigma)$ is a factor.
\end{itemize}
\end{thmb}
\begin{defn}
We say that the pair $(G,\sigma)$ satisfies \emph{condition~K} if (i) is satisfied.
\end{defn}

If $G$ has the icc property, then $(G,\sigma)$ always satisfies condition~K.
If $G$ is abelian, or more generally, if all the conjugacy classes of $G$ are finite,
then $(G,\sigma)$ satisfies condition~K only if there are no nontrivial $\sigma$-regular elements in $G$.

\begin{exmp}\label{klein}
For $n\geq 2$, let $\mathbb{Z}_n$ denote the cyclic group of order $n$.
Then $\mathcal{M}(\mathbb{Z}_n\times\mathbb{Z}_n)\cong\mathbb{Z}_n$ and its elements may be represented by multipliers $\sigma_k$ given by
\begin{displaymath}
\sigma_k((a_1,a_2),(b_1,b_2))=e^{2\pi i\frac{k}{n}a_2b_1}
\end{displaymath}
for $0\leq k\leq n-1$.
An element $a=(a_1,a_2)$ in $\mathbb{Z}_n\times\mathbb{Z}_n$ is $\sigma_k$-regular if and only if both $ka_1$ and $ka_2$ belong to $n\mathbb{Z}$.
Therefore, $(\mathbb{Z}_n\times\mathbb{Z}_n,\sigma_k)$ satisfies condition~K if and only if $k$ and $n$ are relatively prime,
in which case we have
\begin{displaymath}
C^*(\mathbb{Z}_n\times\mathbb{Z}_n,\sigma_k)\cong C^*_r(\mathbb{Z}_n\times\mathbb{Z}_n,\sigma_k)
=W^*(\mathbb{Z}_n\times\mathbb{Z}_n,\sigma_k)\cong M_n(\mathbb{C}).
\end{displaymath}
\end{exmp}

\begin{exmp}\label{n-torus}
It is well known that $\mathcal{M}(\mathbb{Z}^n)\cong\mathbb{T}^{\frac{1}{2}n(n-1)}$ and that the multipliers are, up to similarity, determined by
\begin{displaymath}
\sigma_{\theta}(a,b)=e^{2\pi i\sum_{1\leq i<j\leq n}a_it_{ij}b_j}
\end{displaymath}
for $\theta=(t_{12},t_{13},\dotsc,t_{n-1,n})$ in $[0,1)^{\frac{1}{2}n(n-1)}$.
Note that the $C^*$-algebras associated with the pair $(\mathbb{Z}^n,\sigma_{\theta})$,
$C^*(\mathbb{Z}^n,\sigma_{\theta})\cong C^*_r(\mathbb{Z}^n,\sigma_{\theta})$,
are the noncommutative $n$-tori when $\theta$ is nonzero.

Furthermore, $(\mathbb{Z}^n,\sigma_{\theta})$ satisfies condition~K if there are no nontrivial $\sigma_{\theta}$-regular elements in $\mathbb{Z}^n$, that is,
if there for all $a$ in $\mathbb{Z}^n$ exists $b$ in $\mathbb{Z}^n$ such that
\begin{displaymath}
\sigma_{\theta}(a,b)\overline{\sigma_{\theta}(b,a)}=e^{2\pi i\sum_{1\leq i<j\leq n}t_{ij}(a_ib_j-b_ia_j)}\neq 1.
\end{displaymath}
For $n=2$ and $3$ we can give a good description of this property.
Indeed, $(\mathbb{Z}^2,\sigma_{\theta})$ satisfies condition~K if and only if $\theta$ is irrational,
and $(\mathbb{Z}^3,\sigma_{\theta})$ satisfies condition~K if and only if
\begin{displaymath}
\operatorname{dim}\mathbb{Q}_{\theta}=3\text{ or }4,
\end{displaymath}
where $\mathbb{Q}_{\theta}$ denotes the vector space over $\mathbb{Q}$ spanned by $\{1,t_{12},t_{13},t_{23}\}$.

For $n\geq 4$, the situation is more complicated.
In particular, condition~K on $(\mathbb{Z}^n,\sigma_{\theta})$ does not only depend on the dimension of $\mathbb{Q}_{\theta}$.
For example, if $t_{12}=t_{34}$ is some irrational number in $[0,1)$ and $t_{ij}=0$ elsewhere, then $\operatorname{dim}\mathbb{Q}_{\theta}=2$,
and $(\mathbb{Z}^4,\sigma_{\theta})$ satisfies condition~K.
On the other hand, if $t_{12}=t_{23}=t_{34}=1-t_{14}$ is some irrational number in $[0,1)$ and $t_{13}=t_{24}=0$, then $\operatorname{dim}\mathbb{Q}_{\theta}=2$,
but it can be easily checked that $(1,1,1,1)$ in $\mathbb{Z}^4$ is $\sigma_{\theta}$-regular.
\end{exmp}

\begin{exmp}
For each natural number $n\geq 2$ let $G(n)$ be the group with presentation
\begin{displaymath}
G(n)=\langle u_i,v_{jk} : [v_{jk},v_{lm}]=[u_i,v_{jk}]=e, [u_j,u_k]=v_{jk}\rangle
\end{displaymath}
for $1\leq i\leq n$, $1\leq j<k\leq n$ and $1\leq l<m\leq n$.
The group $G(n)$ is sometimes called the \emph{free nilpotent group of class $2$ and rank $n$}.

In a separate work, we will calculate the multipliers of $G(n)$ and show that
\begin{displaymath}
\mathcal{M}(G(n))\cong\mathbb{T}^{\frac{1}{3}(n-1)n(n+1)}.
\end{displaymath}
Note that $G(2)$ is isomorphic with the discrete Heisenberg group and this case is already investigated by Packer \cite{Packer-H}.

To describe our result in the case of $G(3)$, we first remark that $G(3)$ is isomorphic to the group
with elements $a=(a_1,a_2,a_3,a_4,a_5,a_6)$, where all entries are integers,
and with multiplication defined by
\begin{displaymath}
a\cdot b=(a_1+b_1,a_2+b_2,a_3+b_3,a_4+b_4+a_1b_2,a_5+b_5+a_1b_3,a_6+b_6+a_2b_3).
\end{displaymath}
For every $\mu$ in $\mathbb{T}^8$, the element $[\sigma_{\mu}]$ in $\mathcal{M}(G(3))$ may be represented by
\begin{displaymath}
\begin{split}
\sigma_{\mu}(a,b)&=\mu_{13}^{b_6a_1-b_3a_4}\mu_{22}^{b_5a_2+b_3(a_4+a_1a_2)}\\
&\quad\cdot\mu_{11}^{b_4a_1+\frac{1}{2}b_2a_1(a_1-1)}\mu_{21}^{a_2(b_4+a_1b_2)+\frac{1}{2}a_1b_2(b_2-1)}\\
&\quad\cdot\mu_{12}^{b_5a_1+\frac{1}{2}b_3a_1(a_1-1)}\mu_{32}^{a_3(b_5+a_1b_3)+\frac{1}{2}a_1b_3(b_3-1)}\\
&\quad\cdot\mu_{23}^{b_6a_2+\frac{1}{2}b_3a_2(a_2-1)}\mu_{33}^{a_3(b_6+a_2b_3)+\frac{1}{2}a_2b_3(b_3-1)}
\end{split}
\end{displaymath}
where $\mu_{ij}\in\mathbb{T}$.

The pair $(G(3),\sigma_{\mu})$ satisfies condition~K if and only if $G(3)$ has no nontrivial central $\sigma_{\mu}$-regular elements, that is,
if for all $c=(0,0,0,c_1,c_2,c_3)$ in $Z(G(3))=\mathbb{Z}^3$ there exists $a$ in $G(3)$ such that $\sigma_{\mu}(a,c)\overline{\sigma_{\mu}(c,a)}\neq 1$.

Set $\mu_{31}=\mu_{13}\overline{\mu_{22}}$.
One can then show that this holds if and only if for each nontrivial $c$ in $\mathbb{Z}^3$ there is some $i=1,2$ or $3$ such that
\begin{displaymath}
\prod_{1\leq j\leq 3} \mu_{ij}^{c_j}\neq 1.
\end{displaymath}

\end{exmp}

\section{Primeness and primitivity}\label{pp}

Henceforth, we fix a group $G$ and a multiplier $\sigma$ on $G$.
Consider the right regular $\overline{\sigma}$-projective unitary representation $\rho_{\overline{\sigma}}$ of $G$ on $B(\ell^2(G))$ defined by
\begin{displaymath}
(\rho_{\overline{\sigma}}(a)\xi)(c)=(\xi*_{\overline{\sigma}}\delta_a^*)(c)=\overline{\sigma(c,a)}\xi(ca).
\end{displaymath}
To simplify notation in what follows, we write just $\overline{\rho}$ and $\lambda$ for $\rho_{\overline{\sigma}}$ and $\lambda_{\sigma}$.
It is straightforward to see that $\lambda(a)$ commutes with $\overline{\rho}(b)$ for all $a,b$ in $G$, that is,
$W^*(G,\sigma)$ is contained in $\overline{\rho}(G)'$.
In fact, it is well known that $W^*(G,\sigma)=\overline{\rho}(G)'$.
Moreover,
\begin{equation}\label{conjugation}
(\lambda(a)\overline{\rho}(a)\xi)(c)=\overline{\sigma(a^{-1},c)}\sigma(a^{-1}ca,a^{-1})\xi(a^{-1}ca)
\end{equation}
for all $a,c\in G$ and all $\xi\in\ell^2(G)$.
In particular,
\begin{equation}\label{commutation}
\lambda(a)\overline{\rho}(a)\delta_e=\overline{\rho}(a)\lambda(a)\delta_e=\delta_e
\end{equation}
for all $a\in G$.

\begin{rem}\label{separating}
The vector $\delta_e$ is clearly cyclic for $W^*(G,\sigma)$.
It is also separating.
Indeed, if $x\delta_e=0$, then
\begin{displaymath}
x\delta_a=x\lambda(a)\delta_e=x\overline{\rho}(a)^*\delta_e=\overline{\rho}(a)^*x\delta_e=0
\end{displaymath}
for all $a\in G$.
Moreover, the state $\varphi$ given by $\varphi(x)=\langle x\delta_e,\delta_e\rangle$ is a faithful trace on $W^*(G,\sigma)$.
Thus, $W^*(G,\sigma)$ is finite and is therefore a II$_1$ factor whenever $G$ is infinite and $(G,\sigma)$ satisfies condition~K,
according to \hyperref[thmb]{Theorem~B}.
\end{rem}

\begin{lem}\label{center}
Let $T$ be an operator in $W^*(G,\sigma)$ and set $f_T=T\delta_e$. Then the following are equivalent:
\begin{itemize}
\item[(i)] $T$ belongs to the center of $W^*(G,\sigma)$.
\item[(ii)] $f_T(aca^{-1})=\sigma(a,c)\overline{\sigma(aca^{-1},a)}f_T(c)$ for all $a,c\in G$.
\end{itemize}
Moreover, $f_T$ can be nonzero only on the finite conjugacy classes.
\end{lem}

\begin{proof}
The operator $T$ belongs to the center of $W^*(G,\sigma)$ if and only if $T=\lambda(a)T\lambda(a)^*$ for all $a\in G$.
Since, by Remark~\ref{separating}, $\delta_e$ is separating for $W^*(G,\sigma)$, this is equivalent to
$f_T=\lambda(a)T\lambda(a)^*\delta_e$ for all $a\in G$.
By \eqref{commutation} we have
\begin{displaymath}
\lambda(a)T\lambda(a)^*\delta_e=\lambda(a)T\overline{\rho}(a)\delta_e=\lambda(a)\overline{\rho}(a)T\delta_e=\lambda(a)\overline{\rho}(a)f_T
\end{displaymath}
for all $a\in G$.
Thus $T$ belongs to the center if and only if $f_T=\lambda(a)\overline{\rho}(a)f_T$
for all $a\in G$ and the desired equivalence now follows from \eqref{conjugation}.
If a function $f$ satisfies (ii), then $\lvert f\rvert$ is constant on conjugacy classes.
Since $f_T$ belongs to $\ell^2(G)$, it can be nonzero only on the finite conjugacy classes.
\end{proof}


\begin{rem}
Lemma~\ref{center} 
is proved in \cite[Theorem~1]{Kleppner}.
However, the proof provided above is shorter. 
Lemma~\ref{function} below is proved in \cite[Lemma~2]{Kleppner} in the case where $C$ is a single point.
Also, note that we do not restrict to normalized multipliers as in \cite{Kleppner}.
\end{rem}

\begin{lem}\label{function}
Let $C$ be a conjugacy class of $G$. Then following are equivalent:
\begin{itemize}
\item[(i)] $C$ is $\sigma$-regular.
\item[(ii)] There is a function $f:G\to\mathbb{C}$ satisfying:
\begin{itemize}
\item[1.] $f(c)\neq 0$ for all $c\in C$.
\item[2.] $f(aca^{-1})=\sigma(a,c)\overline{\sigma(aca^{-1},a)}f(c)$ for all $c\in C$ and all $a\in G$.
\end{itemize}
\end{itemize}
Moreover, $f$ can be chosen in $\ell^2(G)$ if and only if $C$ is finite.
\end{lem}

\begin{proof}
(ii) $\Rightarrow$ (i):
Suppose $c$ belongs to $C$ and that $a$ commutes with $c$.
Then there is a function $f:G\to\mathbb{C}$ satisfying $0\neq f(c)=\sigma(a,c)\overline{\sigma(c,a)}f(c)$.
Hence $\sigma(a,c)=\sigma(c,a)$, so $c$ is $\sigma$-regular.

(i) $\Rightarrow$ (ii):
This clearly holds if $C$ is trivial, so suppose $C$ is nontrivial and $\sigma$-regular and fix an element $c$ in $C$.
Define a function $f:G\to\mathbb{C}$ by
\begin{displaymath}
f(x)=
\begin{cases}
\sigma(a,c)\overline{\sigma(aca^{-1},a)}&\text{ if }x\in C,\quad x=aca^{-1}\text{ for some }a\in G\\
0&\text{ if }x\notin C
\end{cases}
\end{displaymath}
First we show that $f$ is well-defined, so assume $aca^{-1}=bcb^{-1}$, and note that
\begin{equation*}
\begin{split}
\sigma(a^{-1},aca^{-1})\sigma(ca^{-1},b)&=\sigma(a^{-1},aca^{-1}b)\sigma(aca^{-1},b)\\
                                     &=\sigma(a^{-1},bc)\sigma(bcb^{-1},b).
\end{split}
\end{equation*}
As $c$ is $\sigma$-regular and commutes with $a^{-1}b$, $\sigma(a^{-1}b,c)=\sigma(c,a^{-1}b)$.
Thus
\begin{equation*}
\begin{split}
\sigma(c,a^{-1})\sigma(ca^{-1},b)&=\sigma(c,a^{-1}b)\sigma(a^{-1},b)\\
                               &=\sigma(a^{-1},b)\sigma(a^{-1}b,c)=\sigma(a^{-1},bc)\sigma(b,c).
\end{split}
\end{equation*}
Together, we get from these equations that
\begin{equation}\label{well-defined}
\sigma(a^{-1},aca^{-1})\sigma(b,c)=\sigma(c,a^{-1})\sigma(bcb^{-1},b).
\end{equation}
Finally, the two identities
\begin{gather*}
\sigma(a^{-1},aca^{-1})\sigma(ca^{-1},a)=\sigma(a^{-1},ac)\sigma(aca^{-1},a)\\
\sigma(c,a^{-1})\sigma(ca^{-1},a)=\sigma(a^{-1},a)=\sigma(a^{-1},ac)\sigma(a,c)
\end{gather*}
give that
\begin{equation}\label{sigma-tilde}
\sigma(a^{-1},aca^{-1})\sigma(a,c)=\sigma(c,a^{-1})\sigma(aca^{-1},a).
\end{equation}
Combining \eqref{well-defined} and \eqref{sigma-tilde} we get that
\begin{equation*}
\sigma(a,c)\overline{\sigma(aca^{-1},a)}=\sigma(b,c)\overline{\sigma(bcb^{-1},b)}.
\end{equation*}
Hence $f$ is well-defined, so $f(aca^{-1})=f(bcb^{-1})$. \bigskip

It is easily seen that $\lvert f(x)\rvert=1$ for all $x$ in $C$.
In fact, if $f$ is any function satisfying (ii), then $\lvert f\rvert$ must be constant and nonzero on $C$,
hence $f$ belongs to $\ell^2(G)$ if and only if $C$ is finite. \bigskip

In particular, $f(c)=1$ in our case, so $f$ satisfies part~2 of (ii) for the chosen $c$ in $C$.
It remains to show that $f$ satisfies part~2 of (ii) for all other $x$ in $C$.
Suppose $x$ is an element of $C$, that is, there exists $b$ in $G$ such that $x=bcb^{-1}$. Note first that
\begin{equation*}
f(x)=f(bcb^{-1})=\sigma(b,c)\overline{\sigma(bcb^{-1},b)}=\sigma(b,c)\overline{\sigma(x,b)}.
\end{equation*}
Next,
\begin{equation*}
\begin{split}
\sigma(axa^{-1},a)\sigma(ax,b)\sigma(ab,c)&=\sigma(axa^{-1},ab)\sigma(a,b)\sigma(ab,c)\\
                                         &=\sigma(axa^{-1},ab)\sigma(a,bc)\sigma(b,c),
\end{split}
\end{equation*}
which, since $xb=bc$, gives that
\begin{equation*}
\begin{split}
\sigma(a,x)\overline{\sigma(x,b)}&=\sigma(a,xb)\overline{\sigma(ax,b)}=\sigma(a,bc)\overline{\sigma(ax,b)}\\
&=\sigma(axa^{-1},a)\sigma(ab,c)\overline{\sigma(axa^{-1},ab)}\overline{\sigma(b,c)}.
\end{split}
\end{equation*}
Hence
\begin{equation*}
\begin{split}
f(axa^{-1})&=f(abcb^{-1}a^{-1})=\sigma(ab,c)\overline{\sigma(abcb^{-1}a^{-1},ab)}\\
          &=\sigma(ab,c)\overline{\sigma(axa^{-1},ab)}=\sigma(a,x)\overline{\sigma(axa^{-1},a)}\sigma(b,c)\overline{\sigma(x,b)}\\
          &=\sigma(a,x)\overline{\sigma(axa^{-1},a)}f(x).
\end{split}
\end{equation*}

\end{proof}

Before stating the main theorem, we recall two results which are part of the folklore of operator algebras.
The first can be shown as sketched in the proof of \cite[Proposition~2.3]{Murphy}, while the second is a rather easy consequence of Urysohn's Lemma.
Remark that together these two results imply that if $A$ is von Neumann algebra, then $A$ is prime if and only if it is a factor.

\begin{prop}\label{factor-prime}
If $A$ is a concrete unital $C^*$-algebra and its bicommutant $A''$ is a factor, then $A$ is prime.
\end{prop}

\begin{prop}\label{trivial-center}
Every prime $C^*$-algebra has trivial center.
\end{prop}


\begin{thm}\label{primeness}
The following conditions are equivalent:
\begin{itemize}
\item[(i)] $(G,\sigma)$ satisfies condition~K.
\item[(ii)] $W^*(G,\sigma)$ is a factor.
\item[(iii)] $C^*_r(G,\sigma)$ is prime.
\item[(iv)] $C^*_r(G,\sigma)$ has trivial center.
\end{itemize}
\end{thm}

\begin{proof}
For completeness, we include the few lines required to prove (i) $\Rightarrow$ (ii):
Suppose $(G,\sigma)$ satisfies condition~K and let $T$ belong to the center of $W^*(G,\sigma)$.
By Lemma~\ref{center} and Lemma~\ref{function}, $f_T$ can be nonzero only on the finite $\sigma$-regular conjugacy classes, hence on $\{e\}$.
So $T\delta_e=f_T(e)\delta_e$, thus $T=f_T(e)I$ as $\delta_e$ is separating for $W^*(G,\sigma)$ by Remark~\ref{separating}. \bigskip

The implications (ii) $\Rightarrow$ (iii) $\Rightarrow$ (iv) follow from Proposition~\ref{factor-prime} and \ref{trivial-center}. \bigskip

(iv) $\Rightarrow$ (i):
Suppose $C$ is a finite nontrivial $\sigma$-regular conjugacy class of $G$.
Let $f$ be a function satisfying part~(ii) of Lemma~\ref{function} and define the operator $T=\sum_{c\in C}f(c)\lambda(c)$.
Then $T$ belongs to the center of $C^*_r(G,\sigma)$. Indeed,
\begin{equation*}
\begin{split}
\lambda(a)T\lambda(a)^*&=\sum_{c\in C}f(c)\lambda(a)\lambda(c)\lambda(a)^*\\
                       &=\sum_{c\in C}f(c)\sigma(a,c)\overline{\sigma(aca^{-1},a)}\lambda(aca^{-1})\\
                      &=\sum_{b\in aCa^{-1}} f(a^{-1}ba)\sigma(a,a^{-1}ba)\overline{\sigma(b,a)}\lambda(b)\\
                       &=\sum_{b\in C} f(a^{-1}ba)\overline{\sigma(a^{-1},b)}\sigma(a^{-1}ba,a^{-1})\lambda(b)\\
                       &=\sum_{b\in C} f(b)\lambda(b)=T
\end{split}
\end{equation*}
for all $a\in G$, where the identity \eqref{sigma-tilde} is used to get the fourth equality.
\end{proof}

The proof of the following corollary goes along the same lines as the one given in \cite[Proposition~2.1]{Murphy} in the untwisted case.

\begin{cor}\label{full-primeness}
If $C^*(G,\sigma)$ is prime, then $(G,\sigma)$ satisfies condition K.
\end{cor}

\begin{proof}
Observe that the set $\{\lambda(a)\}_{a\in G}$ is linear independent in $C^*_r(G,\sigma)$, and the universal property of $C^*(G,\sigma)$
ensures that there is a surjective $^*$-homomorphism $C^*(G,\sigma)\to C^*_r(G,\sigma)$ mapping $i_G(a)$ to $\lambda(a)$.
Hence, $\{i_G(a)\}_{a\in G}$ is also linear independent and has dense span in $C^*(G,\sigma)$.

Therefore, the result follows by replacing $i_G$ with $\lambda$ in the proof of Theorem~\ref{primeness},
and repeating the argument for (iii) $\Rightarrow$ (iv) $\Rightarrow$ (i) word by word.
\end{proof}

\begin{rem}
In general, the center of $C^*(G,\sigma)$ is not easily determined.

However, a slightly stronger version of Corollary~\ref{full-primeness} is known in the untwisted case.
If $C^*(G)$ has trivial center, then $G/H$ is icc
whenever $H$ is a normal subgroup of $G$ satisfying Kazhdan's property~T (see e.g. \cite{Losert}). 
\end{rem}

\begin{cor}[{\cite[Proposition 1.4]{Packer-N}}]\label{Packer}
Assume $G$ is countable and amenable. Then the following conditions are equivalent:
\begin{itemize}
\item[(i)] $(G,\sigma)$ satisfies condition~K.
\item[(ii)] $C^*(G,\sigma)$ is primitive.
\end{itemize}
\end{cor}

\begin{proof}
If $(G,\sigma)$ satisfies condition~K, then $C^*_r(G,\sigma)$ is prime by Theorem~\ref{primeness}.
As $G$ is countable, $C^*_r(G,\sigma)$ is separable and hence primitive by Dixmier's result.
Now, the amenability of $G$ implies that $C^*(G,\sigma)\cong C^*_r(G,\sigma)$, so $C^*(G,\sigma)$ is also primitive.
Finally, (ii) always implies (i) by Corollary~\ref{full-primeness}.
\end{proof}

\begin{rem}\label{SL3Z}
Condition~K on $(G,\sigma)$ does not imply primeness or primitivity of $C^*(G,\sigma)$ in general.
To see this, let $G=\operatorname{SL}(3,\mathbb{Z})$ and $\sigma=1$.
Then, $G$ is countable, icc and satisfies Kazhdan's property~T. In particular, G is nonamenable.
As explained in \cite[Proposition~2.5]{Bedos-Omland}, $C^*(G)$ is not primitive. 

On the other hand, I don't know any example of an uncountable and amenable group such that (i) holds, but not (ii).
\end{rem}

\begin{rem}\label{permutation}
If $G$ is countable and nilpotent, then condition~K on $(G,\sigma)$ is actually equivalent to simplicity of $C^*(G,\sigma)$ \cite[Proposition~1.7]{Packer-N}.
The same is also true if $G$ is finite.

However, this does not hold for all countable, amenable groups.
For example, if $G$ is the group of all finite permutations on a countably infinite set,
then $G$ is countable, amenable and icc, so $C^*(G)$ is primitive and nonsimple.
\end{rem}

\begin{rem}
Note that $C^*_r(\operatorname{SL}(3,\mathbb{Z}))$ is known to be simple \cite{PSLnZ}, so Remark~\ref{SL3Z} and \ref{permutation} show that
primitivity of a full twisted group $C^*$-algebra is in general unrelated to simplicity of the corresponding reduced twisted group $C^*$-algebra.
\end{rem}

\begin{prop}\label{amenability}
The following conditions are equivalent:
\begin{itemize}
\item[(i)] $G$ is amenable.
\item[(ii)] $C^*(G,\sigma)$ is nuclear.
\item[(iii)] $C^*_r(G,\sigma)$ is nuclear.
\item[(iv)] $W^*(G,\sigma)$ is injective.
\end{itemize}
\end{prop}

\begin{proof}
This is well known in the untwisted case.
The result in the twisted case appeared in a preprint by B{\'e}dos and Conti \cite{Bedos-Conti},
but was left out in the final version.
For the convenience of the reader we repeat the argument.
First, (i) $\Rightarrow$ (ii) follows from \cite[Corollary~3.9]{Packer-Raeburn-I}.
The implication (ii) $\Rightarrow$ (iii) holds since every quotient of a nuclear $C\sp*$-algebra is nuclear.
Moreover, the von Neumann algebra generated by a nuclear $C\sp*$-algebra is injective, hence (iii) $\Rightarrow$ (iv).
Finally, if $W\sp*(G,\sigma)$ is injective, it has a hypertrace and thus $G$ is amenable by \cite[Corollary~1.7]{Bedos-hypertraces}, so (iv) $\Rightarrow$ (i).
\end{proof}

According to \cite{Connes}, all injective II$_1$ factors acting on a separable Hilbert space are isomorphic to the hyperfinite II$_1$ factor.
Hence, we get the following corollary.

\begin{cor}
Assume $G$ is countably infinite. Then the following conditions are equivalent:
\begin{itemize}
\item[(i)] $G$ is amenable and $(G,\sigma)$ satisfies condition~K.
\item[(ii)] $C^*(G,\sigma)$ is nuclear and primitive.
\item[(iii)] $C^*_r(G,\sigma)$ is nuclear and primitive.
\item[(iv)] $W^*(G,\sigma)$ is the hyperfinite II$_1$ factor.
\end{itemize}
\end{cor}

\section{Direct products}

Let $G_1$ and $G_2$ be two groups.
A function $f:G_1\times G_2\to\mathbb{T}$ is called a \emph{bihomomorphism} if
\begin{displaymath}
f(a_1b_1,a_2)=f(a_1,a_2)f(b_1,a_2)\quad\text{ and }\quad f(a_1,a_2b_2)=f(a_1,a_2)f(a_1,b_2)
\end{displaymath}
for all $a_1,b_1\in G_1$ and $a_2,b_2\in G_2$.
Let $B(G_1,G_2)$ denote the set of bihomomorphisms $G_1\times G_2\to\mathbb{T}$.
This is a group under pointwise multiplication and is isomorphic with $\operatorname{Hom}(G_1,\operatorname{Hom}(G_2,\mathbb{T}))$.

It is well known (see e.g. \cite{Mackey}) that the Schur multiplier of $G_1\times G_2$ decomposes as
\begin{displaymath}
\mathcal{M}(G_1\times G_2)\cong\mathcal{M}(G_1)\oplus\mathcal{M}(G_2)\oplus B(G_1,G_2).
\end{displaymath}
We will only need to know the following.
Let $(\sigma_1,\sigma_2,f)$ be a triple where $\sigma_1$ and $\sigma_2$ are multipliers on $G_1$ and $G_2$, respectively, and $f$ belongs to $B(G_1,G_2)$.
Then we can define a multiplier $\sigma$ on $G_1\times G_2$ by
\begin{equation}\label{direct-cocycle}
\sigma((a_1,a_2),(b_1,b_2))=\sigma_1(a_1,b_1)\sigma_2(a_2,b_2)f(b_1,a_2)
\end{equation}
for $a_1,b_1\in G_1$ and $a_2,b_2\in G_2$, and it can be shown that every multiplier on $G_1\times G_2$ is similar to such a $\sigma$.
When $\sigma$ is a multiplier on $G_1\times G_2$, we let $\sigma_1$ be the multiplier on $G_1$ defined by
\begin{displaymath}
\sigma_1(a_1,b_1)=\sigma((a_1,e),(b_1,e))
\end{displaymath}
for $a_1,b_1\in G_1$ and call it the restriction of $\sigma$ to $G_1$.
Similarly we can define the restriction $\sigma_2$ of $\sigma$ to $G_2$.

Henceforth, we fix two groups $G_1$ and $G_2$, multipliers $\sigma_1$ on $G_1$ and $\sigma_2$ on $G_2$, and a bihomomorphism $f$ in $B(G_1,G_2)$.
We set $G=G_1\times G_2$ and let $\sigma$ be the multiplier on $G$ defined by \eqref{direct-cocycle}.
Moreover, we write $\sigma=\sigma_1\times\sigma_2$ if $f=1$.

It is convenient to record the identity
\begin{equation}\label{regularity-identity}
\sigma(a,b)\overline{\sigma(b,a)}\cdot f(a_1,b_2)\overline{f(b_1,a_2)}
=\sigma_1(a_1,b_1)\overline{\sigma_1(b_1,a_1)}\cdot \sigma_2(a_2,b_2)\overline{\sigma_2(b_2,a_2)}
\end{equation}
which follows directly from \eqref{direct-cocycle}.
Note also that $C$ is a conjugacy class of $G$
if and only if $C=C_1\times C_2$ where $C_1$ and $C_2$ are conjugacy classes of $G_1$ and $G_2$, respectively.


\begin{prop}\label{f-degeneracy}
The following are equivalent:
\begin{itemize}
\item[(i)] $C^*_r(G,\sigma)$ is prime.
\item[(ii)] For every finite nontrivial conjugacy class $C$ of $G$, there exist $a=(a_1,a_2)$ in $C$ and $b=(b_1,b_2)$ in $G$ such that
at least one of these conditions hold:
\begin{itemize}
\item[1.] $a_1b_1=b_1a_1$ and $f(b_1,a_2)\neq\overline{\sigma_1(a_1,b_1)}\sigma_1(b_1,a_1)$.
\item[2.] $a_2b_2=b_2a_2$ and $f(a_1,b_2)\neq\sigma_2(a_2,b_2)\overline{\sigma_2(b_2,a_2)}$.
\end{itemize}
\end{itemize}
\end{prop}

\begin{proof}
Suppose that condition (ii) does not hold.
Then there is a finite nontrivial conjugacy class $C$ such that both 1. and 2. fail for all $a$ in $C$ and $b$ in $G$.
Hence, $f(b_1,a_2)=\overline{\sigma_1(a_1,b_1)}\sigma_1(b_1,a_1)$ and
$f(a_1,b_2)=\sigma_2(a_2,b_2)\overline{\sigma_2(b_2,a_2)}$ whenever $a=(a_1,a_2)$ is in $C$, $b=(b_1,b_2)$ in $G$ and $b$ commutes with $a$.
Then $C$ is $\sigma$-regular by \eqref{regularity-identity}, and therefore $(G,\sigma)$ does not satisfy condition~K,
that is, $C^*_r(G,\sigma)$ is not prime by Theorem~\ref{primeness}. Thus, (i) $\Rightarrow$ (ii).

Conversely,
assume that $(G,\sigma)$ does not satisfy condition~K and let $C=C_1\times C_2$ be a finite nontrivial $\sigma$-regular conjugacy class of $G$.
If $b_1$ in $G_1$ commutes with $a_1$ in $C_1$, then $(b_1,e)$ commutes with $(a_1,a_2)$ for every $a_2$ in $C_2$.
Hence, the $\sigma$-regularity of $C$ and identity \eqref{regularity-identity} give that
\begin{equation*}
f(b_1,a_2)=\overline{\sigma_1(a_1,b_1)}\sigma_1(b_1,a_1)
\end{equation*}
whenever $a$ belongs to $C$ and $b_1$ in $G_1$ commutes with $a_1$.
Similarly,
\begin{equation*}
f(a_1,b_2)=\sigma_2(a_2,b_2)\overline{\sigma_2(b_2,a_2)}
\end{equation*}
whenever $b_2$ in $G_2$ commutes with $a_2$.
It follows that for all $a$ in $C$ and $b$ in $G$, both 1. and 2. fail to hold, hence condition (ii) is not satisfied.
\end{proof}

\begin{rem}
Let $G_1$ and $G_2$ be abelian and assume that $\sigma_1$ and $\sigma_2$ are trivial.
Condition (ii) of Proposition~\ref{f-degeneracy} then says that
for all nontrivial $(a_1,a_2)$ in $G$ there exists $(b_1,b_2)$ in $G$ such that $f(a_1,b_2)\neq 1$ or $f(b_1,a_2)\neq 1$.
If this holds, $\sigma$ is called nondegenerate and it was first shown by Slawny \cite[Theorem~3.7]{Slawny}
that $C^*(G,\sigma)\cong C^*_r(G,\sigma)$ is simple in this case.
\end{rem}

\begin{lem}\label{f-regularity}
Let $a=(a_1,a_2)$ be an element in $G$. If two of the following conditions hold, then all three hold:
\begin{itemize}
\item[(i)] $a$ is $\sigma$-regular.
\item[(ii)] $a_i$ is $\sigma_i$-regular for both $i=1$ and $2$.
\item[(iii)] $f(a_1,b_2)=f(b_1,a_2)$ whenever $b=(b_1,b_2)$ commutes with $a$.
\end{itemize}
Moreover, (iii) is equivalent to:
\begin{itemize}
\item[(iv)] $f(a_1,b_2)=f(b_1,a_2)=1$ whenever $b=(b_1,b_2)$ commutes with $a$.
\end{itemize}
\end{lem}

\begin{proof}
Suppose that (ii) holds and pick any $b=(b_1,b_2)$ in $G$.
Then it follows readily from \eqref{regularity-identity} that (i) holds if and only if (iii) holds.

Next, assume that (iii) holds and let $b=(b_1,b_2)$ commute with $a$.
Then $b'=(b_1,e)$ also commutes with $a$, so $1=f(a_1,e)=f(b_1,a_2)$.
Similarly, we get $f(a_1,b_2)=1$ and thus (iv) holds.

Suppose finally that (i) and (iii) hold and pick an element $b=(b_1,b_2)$ in $G$ that commutes with $a$.
As (iv) also holds, we have that $f(b_1,a_2)=1$.
By applying \eqref{regularity-identity} with $b'=(b_1,e)$, we see that $a_1$ is $\sigma_1$-regular.
Similarly, $f(a_1,b_2)=1$ and $a_2$ is $\sigma_2$-regular.
\end{proof}

\begin{cor}\label{f-conjugacy-class}
Let $C=C_1\times C_2$ be a conjugacy class of $G$.
Suppose there is some $a=(a_1,a_2)$ in $C$ such that $f(a_1,b_2)=f(b_1,a_2)$ whenever $b=(b_1,b_2)$ commutes with $a$.
Then the following are equivalent:
\begin{itemize}
\item[(i)] $C$ is a finite nontrivial $\sigma$-regular conjugacy class of $G$.
\item[(ii)] $C_i$ is a finite $\sigma_i$-regular conjugacy class of $G_i$ for both $i=1$ and $2$ and at least one of $C_1$ and $C_2$ is nontrivial.
\end{itemize}
\end{cor}

\begin{cor}\label{f-negation}
Suppose both $C^*_r(G_1,\sigma_1)$ and $C^*_r(G_2,\sigma_2)$ are prime.
Let $a=(a_1,a_2)$ be such that $f(a_1,b_2)=f(b_1,a_2)$ whenever $b=(b_1,b_2)$ commutes with $a$.
Then at most one of the following two conditions hold:
\begin{itemize}
\item[(i)] $a$ is $\sigma$-regular.
\item[(ii)] $a$ belongs to a finite nontrivial conjugacy class of $G$.
\end{itemize}
\end{cor}

\begin{cor}
Suppose $f(a_1,b_2)=f(b_1,a_2)$ whenever $a=(a_1,a_2)$ is $\sigma$-regular and $b=(b_1,b_2)$ commutes with $a$.
Then $C^*_r(G,\sigma)$ is prime if both $C^*_r(G_1,\sigma_1)$ and $C^*_r(G_2,\sigma_2)$ are prime.
\end{cor}

\begin{rem}
In general, primeness of $C^*_r(G,\sigma)$ does not imply primeness of either $C^*_r(G_1,\sigma_1)$ or $C^*_r(G_2,\sigma_2)$.
For example, if $G_1=G_2=\mathbb{Z}$, then $C^*(G,\sigma)$ can be simple even if both $\sigma_1$ and $\sigma_2$ are trivial.

Also, $C^*_r(G,\sigma)$ can be nonprime even if both $C^*(G_1,\sigma_1)$ and $C^*(G_2,\sigma_2)$ are simple.
To see this, let $G_1=G_2=\mathbb{Z}^2$ and consider the case described in the last part of Example~\ref{n-torus}.
\end{rem}

\begin{prop}
Suppose $f(c_1,c_2)=1$ whenever $c_i$ belongs to a finite conjugacy class of $G_i$ for either $i=1$ or $2$.
Then $C^*_r(G,\sigma)$ is prime if and only if both $C^*_r(G_1,\sigma_1)$ and $C^*_r(G_2,\sigma_2)$ are prime.

In particular, this holds when $\sigma=\sigma_1\times\sigma_2$.
\end{prop}

\begin{proof}
Suppose $C^*_r(G,\sigma)$ is prime and $C_1$ is a finite $\sigma_1$-regular conjugacy class of $G_1$.
Then $C_1\times\{e\}$ is $\sigma$-regular by Corollary~\ref{f-conjugacy-class} so $C_1=\{e\}$ and hence $C^*_r(G_1,\sigma_1)$ is prime.
Similarly we get that $C^*_r(G_2,\sigma_2)$ is prime.

The converse follows directly from Corollary~\ref{f-negation}.
\end{proof}


\begin{rem}
Assume that $\sigma=\sigma_1\times\sigma_2$.
Then $C^*_r(G,\sigma)$ is simple if and only both $C^*_r(G_1,\sigma_1)$ and $C^*_r(G_2,\sigma_2)$ are simple.
Indeed, note that the map $\lambda_{\sigma}(a_1,a_2)\mapsto \lambda_{\sigma_1}(a_1)\otimes\lambda_{\sigma_2}(a_2)$ induces an isomorphism
\begin{displaymath}
C^*_r(G,\sigma)\cong C^*_r(G_1,\sigma_1)\otimes_{\textsf{min}}C^*_r(G_2,\sigma_2).
\end{displaymath}
The result now follows from the fact that a spatial tensor product of two $C^*$-algebras is simple if and only if
both involved $C^*$-algebras are simple (see \cite[11.5.5-6]{Kadison-Ringrose}).
\end{rem}

The only positive result on primitivity so far in this paper concerns countable, amenable groups.
However, Corollary~\ref{Packer} relies on Dixmier's result that is not constructive
in the sense that it does not give a procedure to construct an explicit faithful irreducible representation.

In some cases, one may construct faithful irreducible representations of $C^*(G,\sigma)$
through an inducing process on representations of $C^*(G_1,\sigma_1)$.

\begin{thm}\label{primitivity-repr}
Assume that $G_2$ is amenable.
Suppose there exists a faithful irreducible representation $\pi$ of $C^*(G_1,\sigma_1)$ such that for any given nontrivial $a_2$ in $G_2$,
there exists $a_1$ in $G_1$ such that
\begin{displaymath}
f(a_1,a_2)\pi(i_{G_1}(a_1))\not\simeq\pi(i_{G_1}(a_1)).
\end{displaymath}
Then $C^*(G,\sigma)$ is primitive.
\end{thm}

\begin{proof}
Recall that there is a twisted action $(\alpha,\omega)$ of $G_2$ on $A=C^*(G_1,\sigma_1)$ satisfying (see e.g. \cite{Zeller-Meier}) 
\begin{displaymath}
\begin{split}
\alpha_{a_2}(i_{G_1}(a_1))&=\overline{f(a_1,a_2)}i_{G_1}(a_1),\\
\omega(a_2,b_2)&=\sigma_2(a_2,b_2).
\end{split}
\end{displaymath}
Hence, there is also a natural action of $G_2$ on the set $\widehat{A}{^0}$ of equivalence classes of faithful irreducible representations of $A$ given by
\begin{displaymath}
a_2\cdot [\psi]=[\psi\circ\alpha_{a_2^{-1}}].
\end{displaymath}
For any given nontrivial $a_2$ in $G_2$, the assumptions on $\pi$ gives that
\begin{displaymath}
\pi\circ\alpha_{a_2^{-1}}(i_{G_1}(a_1))=f(a_1,a_2)\pi(i_{G_1}(a_1))\not\simeq\pi(i_{G_1}(a_1))
\end{displaymath}
for some $a_1$ in $G_1$.
Hence
\begin{displaymath}
a_2\cdot [\pi]\neq [\pi]
\end{displaymath}
for all nontrivial $a_2$ in $G_2$. In other words, $[\pi]$ is a free point for this action.
The conclusion follows from \cite[Theorem~2.1]{Bedos-Omland}.
\end{proof}

\begin{exmp}
Let $G=\mathbb{F}_2\times\mathbb{Z}$ and let $u,v$ be the generators of $\mathbb{F}_2$.
Since $\mathcal{M}(\mathbb{F}_2)=\mathcal{M}(\mathbb{Z})=\{1\}$, every multiplier on $G$ is, up to similarity,
determined by a bihomomorphism $f:\mathbb{F}_2\times\mathbb{Z}\to\mathbb{T}$.
Moreover, $f$ is determined by its values on the generators, that is, by $f(u,1)$ and $f(v,1)$.
Let $\sigma$ be the multiplier on $G$ defined by these two numbers, say $\mu$ and $\nu$.
We remark that
\begin{displaymath}
C^*(G,\sigma)\cong C^*(\mathbb{F}_2)\rtimes_{\alpha}\mathbb{Z}
\end{displaymath}
where $\alpha$ is determined by $\alpha_k(i_{\mathbb{F}_2}(x))=\overline{f(x,k)}i_{\mathbb{F}_2}(x)$ for $x\in\mathbb{F}_2$ and $k\in\mathbb{Z}$. \bigskip

Assume $\mu$ is nontorsion and let $A=C^*(\mathbb{F}_2)$ sit inside $B(\mathcal{H})$ for some separable Hilbert space $\mathcal{H}$.
Let $U=i_{\mathbb{F}_2}(u)$ and $V=i_{\mathbb{F}_2}(v)$ be the two unitaries in $B(\mathcal{H})$ generating $A$.
Now, proceeding as Choi in \cite[Lemma~4]{Choi}, there is an operator $D$ for which $U-D$ is compact and such that the following hold;
with respect to a suitable basis on $\mathcal{H}$, $D$ is diagonal with diagonal entries $\{z_i\}_{i=1}^{\infty}$
satisfying $|z_i|=1$ for all $i$, $z_1=1$, $z_i\neq z_j$ if $i\neq j$ and $z_i\notin\{\mu^k:k\in\mathbb{Z}\}$ when $i\geq 2$.

Using \cite[Lemma~5]{Choi}, we can find a compact perturbation $E$ of $V$
which is a unitary operator having no common nontrivial invariant subspace with $D$.
Then, as explained in \cite[Theorem~6]{Choi}, the map $U\mapsto D$,
$V\mapsto E$ defines a faithful and irreducible representation $\pi$ of $A$ on $\mathcal{H}$.

Now we have
\begin{displaymath}
\pi\circ\alpha_{k^{-1}}(U)=f(u,k)\pi(U)=\mu^k\pi(U)\not\simeq\pi(U)
\end{displaymath}
for all $k$ in $\mathbb{Z}$.
Indeed, this holds as the point spectrum of $\pi(U)=D$ is different from the point spectrum of $\pi(\alpha_{k^{-1}}(U))=\mu^kD$ by construction.

A similar argument also holds if $\nu$ is nontorsion.
Hence, we get from Theorem~\ref{primitivity-repr} that $C^*(G,\sigma)$ is primitive if either $\mu$ or $\nu$ is nontorsion.

On the other hand, if $(G,\sigma)$ satisfies condition~K, then at least one of $\mu$ and $\nu$ must be nontorsion,
so this is also a necessary condition for primitivity of $C^*(G,\sigma)$.
Indeed, condition (ii) of Proposition~\ref{f-degeneracy} does not hold if both $\mu$ and $\nu$ are torsion.

\end{exmp}

\begin{prop}
Assume that $\sigma=\sigma_1\times\sigma_2$ and that both $C^*(G_1,\sigma_1)$ and $C^*(G_2,\sigma_2)$ are primitive.
Then $C^*(G,\sigma)$ is primitive if at least one of $G_1$ and $G_2$ is amenable.
\end{prop}

\begin{proof}
Without loss of generality we may assume that $G_1$ is amenable.
Then $C^*(G_1,\sigma_1)$ is nuclear by Proposition~\ref{amenability}
so the minimal and maximal tensor products of $C^*(G_1,\sigma_1)$ and $C^*(G_2,\sigma_2)$ coincide.
According to \cite[Section~3]{Guichardet},
there is a unique isomorphism
\begin{displaymath}
C^*(G,\sigma)\to C^*(G_1,\sigma_1)\otimes C^*(G_2,\sigma_2)
\end{displaymath}
given by $i_G(a_1,a_2)\mapsto i_{G_1}(a_1)\otimes i_{G_2}(a_2)$.

For $i=1,2$, let $\pi_i$ be a faithful irreducible representation of $C^*(G_i,\sigma_i)$ on $\mathcal{H}_i$.
Then $\pi=\pi_1\otimes\pi_2$ is a representation of $C^*(G,\sigma)$ on $\mathcal{H}=\mathcal{H}_1\otimes\mathcal{H}_2$,
which is faithful by \cite[Theorem~6.5.1]{Murphy-book} and irreducible by \cite[Section~2]{Guichardet}.
Hence $C^*(G,\sigma)$ is primitive.
\end{proof}

\begin{rem}
Primitivity of $C^*(G,\sigma)$ is in general difficult to decide.
For example, let $\mathbb{F}$ be a free nonabelian group, then it is unknown whether $C^*(\mathbb{F}\times\mathbb{F})$ is primitive
(see \cite[Remark~2.2]{Bedos-Omland} for a brief discussion).
\end{rem}

\section{Free products}

In some sense, free products are easier to treat than direct products, since the Schur multiplier decomposes nicely.
Indeed, let $G_1$ and $G_2$ be two groups.
Then we have that (see e.g. \cite[page~51]{Brown})
\begin{equation}\label{free-product}
\mathcal{M}(G_1 * G_2)\cong\mathcal{M}(G_1)\oplus\mathcal{M}(G_2).
\end{equation}

Let $\sigma_1$ be a normalized multiplier on $G_1$ and $\sigma_2$ a normalized multiplier on $G_2$.
Following \cite[Section~5]{McClanahan}, we will explain how to obtain a normalized free product multiplier $\sigma_1 * \sigma_2$ on $G_1*G_2$.

Every nontrivial element $x$ in $G_1*G_2$ can be uniquely written as a reduced word $x=x_1x_2\dotsm x_n$,
for which the letters with odd index belong to $G_i$ and the letters with even index belong to $G_j$ for $i\neq j$.
Define the length function as $l(x)=l(x_1x_2\dotsm x_n)=n$ and $l(e)=0$.
If $l(x),l(y)\leq 1$, we write $x\perp y$ if either $x=e$ or $y=e$ or else if $x$ is in $G_i$ and $y$ is in $G_j$ for $i\neq j$.

Let $s(x)$ and $r(x)$ denote the first and last letter of a nontrivial word $x$ and set $s(e)=r(e)=e$.
For a pair of words $(x,y)$, we say that the pair is reduced if $r(x)\neq s(y)^{-1}$.

When $(x,y)$ is not reduced, let $w$ be the longest word such that $r(xw^{-1})\perp s(w)$ and $r(w^{-1})\perp s(wy)$.
Set $x_w=xw^{-1}$ and $y_w=wy$, so that $x=x_ww$ and $y=w^{-1}y_w$.
Let $(x,y)_w=(x_w,y_w)$ be the reduction of $(x,y)$ and note in particular that $x_wy_w=xy$.

If the pair $(x,y)$ is reduced, then we set $(x,y)_w=(x,y)$. \bigskip

Define now the multiplier $\tau$ on $G_1*G_2$ by
\begin{displaymath}
\tau(x,y)=\tau((x,y)_w)=\begin{cases}
\sigma_1(r(x_w),s(y_w)) & \text{ if }r(x_w),s(y_w)\in G_1\setminus\{e\},\\
\sigma_2(r(x_w),s(y_w)) & \text{ if }r(x_w),s(y_w)\in G_2\setminus\{e\},\\
1 & \text{ if }r(x_w)\perp s(y_w),\end{cases}
\end{displaymath}
and note that this definition coincides with the one explained in \cite[Section~5]{McClanahan}.
Furthermore, let
\begin{displaymath}
X=\{[a,b]=aba^{-1}b^{-1}: a\in G_1\setminus\{e\}, b\in G_2\setminus\{e\}\}
\end{displaymath}
and recall that the free nonabelian group on $X$, denoted $\mathbb{F}_X$, may be identified with the normal subgroup of $G_1*G_2$ generated by $X$.

Moreover, define a function $\beta:G_1 * G_2\to\mathbb{T}$ by $\beta(x)=1$ if $x\notin\mathbb{F}_X$,
while for nontrivial $x=q_1^{p_1}\dotsm q_n^{p_n}$ in $\mathbb{F}_X$, where $q_i$ belongs to $X$ and $p_i$ is an integer, we set
\begin{displaymath}
\beta(x)=\beta(q_1^{p_1}\dotsm q_n^{p_n})=\begin{cases}\tau(q_1^{p_1},q_2^{p_2})\tau(q_2^{p_2},q_3^{p_3})\dotsm\tau(q_{n-1}^{p_{n-1}},q_n^{p_n})&\text{ if }n\geq 2,\\
1&\text{ if }n=1.\end{cases}
\end{displaymath}
Now define the multiplier $\sigma$ on $G_1*G_2$ by
\begin{displaymath}
\sigma(x,y)=\beta(x)\beta(y)\overline{\beta(xy)}\tau(x,y).
\end{displaymath}

We write $\sigma=\sigma_1 * \sigma_2$ and note that $\sigma\sim\tau$, $\sigma_{|G_i\times G_i}=\sigma_i$ and $\sigma_{|\mathbb{F}_X\times\mathbb{F}_X}=1$.

On the other hand, if $\sigma$ is a normalized multiplier on $G_1*G_2$, we can define the restriction $\sigma_1$ on $G_1$ by
\begin{displaymath}
\sigma_1(x,y)=\begin{cases}\sigma(x,y)&\text{ if }x,y\in G_1\setminus\{e\},\\ 1&\text{ if $x$ or $y=e$}.\end{cases}
\end{displaymath}
Similarly, we can define the restriction $\sigma_2$ of $\sigma$ to $G_2$.
Next, define the function $\beta:G_1*G_2\to\mathbb{T}$ by $\beta(x)=1$ if $l(x)\leq 1$ and else
\begin{displaymath}
\beta(x)=\beta(x_1\dotsm x_n)=\sigma(x_1,x_2)\sigma(x_1x_2,x_3)\dotsm\sigma(x_1\dotsm x_{n-1},x_n).
\end{displaymath}
Then $\sigma$ is similar to $\sigma_1*\sigma_2$ through $\beta$.

Remark that every multiplier is similar to a normalized one.
Therefore, 
every multiplier on $G_1*G_2$ is similar to $\sigma_1*\sigma_2$ for some normalized multipliers $\sigma_1$ on $G_1$ and $\sigma_2$ on $G_2$.


We are now ready to prove the twisted version of \cite[Theorem~1.2]{Omland-Bedos}.

\begin{thm}\label{primitivity-free}
Assume $G=G_1*G_2$, where $G_1$ and $G_2$ are countable and amenable and $(|G_1|-1)(|G_2|-1)\geq 2$, and let $\sigma$ be a multiplier on $G$.
Then $C^*(G,\sigma)$ is primitive.
\end{thm}

\begin{proof}
We may assume that $\sigma = \sigma_1 * \sigma_2$ where $\sigma_1$ and $\sigma_2$ are normalized multipliers on $G_1$ and $G_2$, respectively,
and that $\sigma_{|\mathbb{F}_X\times\mathbb{F}_X}=1$.
The proof is only a slight modification of the proof of \cite[Theorem~1.2]{Omland-Bedos},
so we just point out what needs to be adjusted in this proof and use the notation therein.
First, recall that there is a twisted action $(\alpha,\omega)$ of $(G_1*G_2)/\mathbb{F}_X\cong G_1\times G_2$ on $H=\mathbb{F}_X$.
Straightforward calculations give that
\begin{displaymath}
\alpha_{(c,d)}(i_H([a,b]))=\begin{cases}i_H(cd[a,b]d^{-1}c^{-1})\cdot\sigma_2(d,b)&\text{ if }d\neq e\\
i_H(cd[a,b]d^{-1}c^{-1})\cdot\sigma_1(c,a)&\text{ if }d=e
\end{cases}
\end{displaymath}
for $a,c\in G_1$ and $b,d\in G_2$.
Hence the expressions in the equations \cite[(2.3),(2.4)]{Omland-Bedos} remain unchanged, so it is enough to reconsider \cite[Case~3]{Omland-Bedos}. 
More straightforward calculations give that the conditions at the bottom of \cite[page~54]{Omland-Bedos} must be replaced with:
\begin{displaymath}
\begin{split}
k&=(s_0,t)\text{ and }k=(s_l,e_2)\text{ if}\\&\qquad\lambda(s_0s_l,t)U(s_0s_l,t)\not\simeq\sigma_1(s_l,s_0s_l)U(s_0,t)(\lambda(s_l,t)U(s_l,t))\sp* \, ;\\
k&=(s_0,e_2)\text{ and }k=(s_l,t)\text{ if}\\&\qquad\lambda(s_0s_l,t)U(s_0s_l,t)\not\simeq\sigma_1(s_0,s_0s_l)\lambda(s_l,t)U(s_l,t)U(s_0,t)\sp* \, ;\\
k&=(s_0,t)\text{ and }k=(s_0s_l,e_2)\text{ if}\\&\qquad\lambda(s_l,t)U(s_l,t)\not\simeq\sigma_1(s_0s_l,s_l)U(s_0,t)(\lambda(s_0s_l,t)U(s_0s_l,t))\sp* \, ;\\
k&=(s_0s_l,t)\text{ and }k=(s_0,e_2)\text{ if}\\&\qquad\lambda(s_l,t)U(s_l,t)\not\simeq\sigma_1(s_0,s_l)\lambda(s_0s_l,t)U(s_0s_l,t)U(s_0,t)\sp* \, .
\end{split}
\end{displaymath}
Now it is easily seen that the rest of the proof works with appropriate modifications.
\end{proof}

\begin{rem}
Theorem~\ref{primitivity-free} is not surprising.
In fact, I am not aware of any pair $(G,\sigma)$ such that $C^*(G)$ is primitive, but $C^*(G,\sigma)$ is nonprimitive.
\end{rem}

\begin{rem}
Let $G=G_1*G_2$, let $\sigma$ be a multiplier on $G$ and assume $\sigma=\sigma_1*\sigma_2$.
Then it is known that (see \cite[Section~5]{McClanahan})
\begin{displaymath}
C^*(G,\sigma) = C^*(G_1,\sigma_1) * C^*(G_2,\sigma_2).
\end{displaymath}
\end{rem}

\begin{exmp}
As explained in Example~\ref{klein} we have that for each natural number $n$,
there exists a multiplier $\sigma_k$ on $\mathbb{Z}_n\times\mathbb{Z}_n$ such that $C^*(\mathbb{Z}_n\times\mathbb{Z}_n,\sigma_k)\cong M_n(\mathbb{C})$.
One immediate consequence of Theorem~\ref{primitivity-free} is that
\begin{displaymath}
M_j(\mathbb{C})*M_k(\mathbb{C})
\end{displaymath}
is primitive for all $j,k\geq 2$.
More generally, it has recently been shown \cite{Dykema-Ayala} that $F_1*F_2$ is primitive whenever $F_1$ and $F_2$
are finite-dimensional $C^*$-algebras and $(\operatorname{dim}{F_1}-1)(\operatorname{dim}{F_2}-1)\geq 2$.
\end{exmp}

\bibliographystyle{plain}

\addcontentsline{toc}{section}{References}

Department of Mathematical Sciences, Norwegian University of Science and Technology (NTNU), NO-7491 Trondheim, Norway.

E-mail address: \href{mailto:tron.omland@math.ntnu.no}{\nolinkurl{tron.omland@math.ntnu.no}}

\end{document}